\documentclass[11pt]{article}

\usepackage{diagbox}
\usepackage{mathtools}
\usepackage{bbm}
\usepackage{latexsym}
\usepackage{epsfig}
\usepackage{amsmath,amsthm,amssymb,enumerate}
\usepackage[hidelinks]{hyperref}
\usepackage[active]{srcltx}
\usepackage{color}
\newcommand{\red}[1]{#1}

\newcounter{rot}

\def\a{\alpha} \def\b{\beta}  
\def\e{\epsilon} \def\f{\phi}   \def\g{\gamma}
\def\G{\Gamma}  
\def\z{\zeta}     \def\l{\lambda}
\def\La{\Lambda}   
   
   \def\U{\Upsilon}

\def\cT{{\cal T}}

\newtheorem{theorem}{Theorem}
\newtheorem{lemma}[theorem]{Lemma}

\newcommand{\hW}[1]{\widehat{W}^{#1}}

\def\cT{{\mathcal T}}

\newcommand{\brac}[1]{\left(#1\right)}

\newcommand{\bfrac}[2]{\left(\frac{#1}{#2}\right)}

\newcommand{\set}[1]{\left\{#1\right\}}

\def\E{\mbox{{\bf E}}}

\def\Pr{\mbox{{\bf Pr}}}

\newcommand{\ignore}[1]{}

\newcommand{\beq}[2]{\begin{equation}\label{#1}#2\end{equation}}

\newcommand{\multstar}[1]{\begin{multline*}#1\end{multline*}}
\newcommand{\mult}[2]{\begin{multline}\label{#1}#2\end{multline}}
\usepackage{tikz}
\usetikzlibrary{arrows}
\usetikzlibrary{decorations}
\usetikzlibrary{shapes.misc}

\newcommand{\dd}{\mathrm{d}}
\newcommand{\1}{\textbf{1}}
\def\hc{\hat{c}_0}
\def\hT{\hat{T}}
\def\hW{\hat{W}}
\def\tT{\tilde{T}}

\begin{document}

\title{A randomly weighted minimum spanning tree with a random cost constraint}
\author{Alan Frieze\thanks{Research supported in part by NSF grant DMS1661063}  and  Tomasz Tkocz\\Department of Mathematical Sciences\\Carnegie Mellon University\\Pittsburgh PA15217\\U.S.A.}
\maketitle
\begin{abstract}
We study the minimum spanning tree problem on the complete graph $K_n$ where an edge $e$ has a weight $W_e$ and a cost $C_e$, each of which is an independent copy of the random variable $U^\g$ where $\g\leq 1$ and $U$ is  the uniform $[0,1]$ random variable. There is also a constraint that the spanning tree $T$ must satisfy $C(T)\leq c_0$. We establish, for a range of values for $c_0,\g$, the asymptotic value of the optimum weight via the consideration of a dual problem. 
\end{abstract}

\bigskip

\begin{footnotesize}

\end{footnotesize}

\bigskip
\section{Introduction}
Let $U$ denoe the uniform $[0,1]$ random variable and let $0<\g\leq 1$. We consider the minimum spanning tree problem in the context of the complete digraph $\vec{K}_n$ where each edge has an independent copy of $U^\g$ for weight $W_e$ and an independent copy of $U^\g$ for cost $C_e$. Let $\cT$ denote the set of spanning trees of $\vec{K}_n$. The weight of a spanning tree $A$ is given by $W(T)=\sum_{e\in T}W_e$ and its cost $C(T)$ is given by $C(T)=\sum_{e\in T}C_e$. The problem we study is
\beq{prob}{
\text{Minimise }W(T)\text{ subject to }T\in\cT,\,C(T)\leq c_0,
}
where $c_0$ may depend on $n$. We let $W^*=W^*(c_0)=W(T^*)$ denote the optimum value to \eqref{prob}.

The unconstrained case of this question ($\g=1,c_0=\infty$) has been well studied: Frieze  \cite{F1}, Steele \cite{Ste1}, Janson \cite{Janson}, Penrose \cite{Pe}, Frieze and McDiarmid \cite{FM}, Frieze, Ruszink\'o and Thoma \cite{FRT}, Beveridge, Frieze and McDiarmid \cite{BFM}, Li and Zhang \cite{LZ} and Cooper, Frieze, Ince, Janson and Spencer \cite{CFIJS} and is well understood. For example, \cite{CFIJS} proves that if $L_n$ denotes the expected minimum weight of a spanning tree then
\[
L_n=\zeta(3)+\frac{c_1}{n}+\frac{c_2+o(1)}{n^{4/3}}
\]
for explicitly defined $c_1,c_2$. Here and throughout, $\zeta(s) = \sum_{k=1}^\infty k^{-s}$ is the zeta function. 

Equation \eqref{prob} defines a natural problem that has been considered in the literature, in the worst-case rather than the average case. See for example Aggarwal, Aneja and Nair \cite{AAN} and Guignard and Rosenwein \cite{GR} (for a directed version) and Goemans and Ravi \cite{GRa}.

We first consider the simpler case where $\g=1$. We need to make the following definitions:
\begin{equation}\label{eq:c1}
c_1 = \frac{1}{\sqrt{2}}\sum_{k=1}^{\infty} \frac{1}{k^{3/2}}\frac{\Gamma\left(k-\frac{1}{2}\right)}{k!}.
\end{equation}
\beq{f}{
f(\b)=\sum_{k=1}^\infty \frac{k^{k-2}}{k!}f_k(\b)
}
where
\beq{fk}{
f_k(\b)=\b^{1/2}\int_{x=0}^\b x^{k-3/2}e^{-kx}\dd x+\int_{x=\b}^\infty x^{k-1}e^{-kx}\dd x, \qquad \beta \geq 0
}
and $\Gamma(s) = \int_0^\infty x^{s-1}e^{-x}\dd x$ is the gamma function.
\begin{theorem}\label{th1}
The following hold w.h.p.:
\begin{enumerate}[(1)]
\item  If
\beq{c0}{
c_0\in \left[c_1(500\log n)^{1/2},\frac{c_1n}{(8000\log n)^{1/2}}\right]
}
then
\beq{main}{
W^*\approx \frac{c_1^2n}{4c_0}.
}
\item Suppose now that $c_0=\a n$ where $\a=O(1)$ is a positive constant. 
\begin{enumerate}[(i)]
\item If $\a>1/2$ then 
\[
W^*\approx \z(3)=\sum_{k=1}^\infty\frac{1}{k^3}.
\]
\item If $0<\a\leq 1/2$ and if $\b^*=\b^*(\a)$ is the solution to 
\beq{beta}{
f'(\b)=2\a,
}
\end{enumerate}
then
\beq{Wstar1}{
W^*\approx f(\b^*)-2\a\b^*.
}
\item Suppose now that $c_0=\a$ where $\a=O(1)$ is a positive constant. 
\begin{enumerate}[(i)]
\item If $\a<\z(3)$ then there is no feasible solution to \eqref{prob}.
\item If $\a>\z(3)$ and if $\b^*=\b^*(\a)$ is the solution to 
\beq{f-}{
f(\b)-\b f'(\b)=\a,
}
then
\beq{f++}{
W^*\approx \frac{f(\b^*)-\a}{2\b^*}n.
}
\end{enumerate}
\end{enumerate}
\end{theorem}

For the case $\g<1$ we will prove the following.
\begin{theorem}\label{thm:main-resx}
Suppose that
\beq{c0range}{
n^{1-\g}\log^{\g/2}n\ll c_0\ll \frac{n}{\log^{\g/2}n}.
}
Then the following holds w.h.p.
\beq{opta}{
W^*\approx \frac{C_\g^2n^{2-\g}}{4c_0},
}
where $C_\g=\frac{\g}{2}\frac{\Gamma(2/\g+1)^{\g/2}}{\Gamma(1/\g+1)^\g} \sum_{k=1}^\infty \frac{\Gamma(k+\g/2-1)}{k^{\g/2+1}k!}$.
\end{theorem}
Note that $C_1=c_1$ and this implies that the expression in \eqref{opta} is consistent with the expression in \eqref{main}.

We will first concentrate on the case $\g=1$. After this, we will continue with the proof of Theorem \ref{thm:main-resx}. We note that a preliminary version containing the results for the case $\g=1$ appeared in \cite{SODA}. The weights and costs will therefore be uniform $[0,1]$ until we reach the more general case in Section \ref{g<1}. We will then prove Theorem \ref{thm:main-resx} as stated and then show how to extend this result to a wider class of distribution via a simple coupling argument from Janson \cite{Jan}.

\section{Outline Proof for $\g=1$}
We tackle \eqref{prob} by considering the dual problem:
\beq{dual}{
\text{Maximise }\f(\l)\text{ over }\l\geq 0,\text{ where }\f(\l)=\min\set{W(T)+\l(C(T)-c_0):\text{$T\in\cT$ }}.
}
We note that
\beq{dual1}{
\text{if $\l\geq 0$ and $T$ is feasible for \eqref{prob} then $\f(\l)\leq W(T)$.}
}
We will show that w.h.p. 
\beq{show}{
\text{that if $\l^*$ solves \eqref{dual} and $T^*$ solves \eqref{prob} then $\f(\l^*)\approx W(T^*)$. }
}
Here $A\approx B$ is an abbreviation for $A=(1+o(1))B$ as $n\to\infty$, assuming that $A=A(n),B=B(n)$.

We use a standard integral formula to compute $\f(\l)$ in Section \ref{secexpect}. This is straightforward, but lengthy. We then prove concentration around the mean in Section \ref{seconc}. We then use a result of \cite{GR} to show in Section \ref{secthm} that in the cases discussed, the duality gap is negligible w.h.p.
\subsection{Consistency in Theorem \ref{th1}}
Before continuing, we will check that the claims in Cases (2) and (3) are intuitively reasonable. First consider Case (2). If $\a>1/2$ and if $T^*$ is the tree minimising $W(T)$ then w.h.p. $W(T^*)\approx \z(3)$ and $C(T^*)\leq (1+o(1))n/2$.

We observe next that $f'(\b)>0$. This follows directly from 
\beq{f'b}{
f'(\b)=\frac12 \sum_{k=1}^\infty \frac{k^{k-2}}{k!}\beta^{-1/2}\int_0^\beta x^{k-3/2}e^{-kx} \dd x.
}
It is shown in an appendix that 
\beq{f''}{
f'(\b)\text{ is a strictly monotone decreasing function.}
}
{\red As such $f$ has a Lipschitz continuous inverse.} By inspection we see that $f'(\infty)=0$.

Note also that $f'(0)=1$  (use L'H\^opital's rule) and
\[
f(0)=\sum_{k=1}^\infty \frac{k^{k-2}}{k!}\int_{x=0 }^\infty x^{k-1}e^{-kx}\dd x= \sum_{k=1}^\infty \frac{k^{k-2}}{k!}\cdot\frac{(k-1)!}{k^k}=\z(3)
\]
and so \eqref{beta} and \eqref{Wstar1} are consistent with (i) when $\a=1/2$.

If $\a<1/2$ then from the above properties of $f'$ we see that \eqref{beta} has a unique positive solution. We derive expression \eqref{Wstar1} below.

Now consider Case (3). If $\a<\z(3)$ then w.h.p. there is no tree $T$ with $C(T)<\a$. If $g(\b)=f(\b)-\b f'(\b)$, then $g(0) = \zeta(3)$, $g'(\b)=-\b f''(\b)>0$ and 
\[
g(\b)\geq\frac{\b^{1/2}}{2} \sum_{k=1}^\infty \frac{k^{k-2}}{k!}\int_{x=0}^\b x^{k-3/2}e^{-kx}\dd x\to\infty\text{ as }\b\to\infty.
\]
This implies that \eqref{f-} has a unique positive solution. We derive expression \eqref{f++} below.

\section{Evaluation of the dual problem}
\subsection{Expectation}\label{secexpect}
\begin{lemma}\label{lm:Elength}
Let $\lambda \geq 0$ and let {\red $L_n=L_n(\l)$} be the total weight of a minimum spanning tree in the complete graph on $n$ vertices with each edge $e$ having weight $Z_e = W_e + \lambda C_e$, where $W_e$ and $C_e$ are i.i.d. random variables uniform on $[0,1]$. We have
\begin{enumerate}[{\red a.}]
\item If $\frac{2000\log n}{n} \leq \lambda \leq \frac{n}{2000\log n}$, then
\begin{equation}\label{eq:EL-formula-lambda-mid}
\E L_n \approx c_1\sqrt{\lambda n}.
\end{equation}
\item If $\lambda < \frac{2000\log n}{n}$, then
\begin{equation}\label{eq:EL-formula-lambda-small}
\E L_n\approx \sum_{k=1}^{\infty}\frac{k^{k-2}}{k!}\Big[\sqrt{\frac{\lambda n}{2}}\int_{0}^{\frac{\lambda n}{2}}x^{k-3/2}e^{-kx} \dd x + \int_{\frac{\lambda n}{2}}^{\infty}x^{k-1}e^{-kx} \dd x\Big].
\end{equation}
\item If $\lambda > \frac{n}{2000\log n}$, then
\begin{equation}\label{eq:EL-formula-lambda-large}
\E L_n \approx \sum_{k=1}^{\infty}\frac{k^{k-2}}{k!} \lambda\Big[\sqrt{\frac{n}{2\lambda}}\int_{0}^{\frac{n}{2\lambda}}x^{k-3/2}e^{-kx} \dd x + \int_{\frac{n}{2\lambda}}^{\infty}x^{k-1}e^{-kx} \dd x\Big].
\end{equation}
\end{enumerate}
The implied $o(1)$ terms in the above expressions can be taken to be independent of $\l$. Also, we have not optimised all constants.
\end{lemma}
\begin{proof}
Let $T$ be a minimum spanning tree. The starting point is Janson's formula \cite{Janson},
\beq{formula}{
\E L_n = \E \sum_{e \in T} Z_e = \E \sum_{e \in T}\int_{0}^\infty \1_{\{Z_e \geq p\}} \dd p = \int_0^\infty\E |\{e \in T, \ Z_e \geq p\}| \dd p = \int_0^\infty \E \big(\kappa(G) - 1\big) \dd p,
}
where $\kappa(G)$ is the number of components in the {\red random} graph $G$ on $n$ vertices with the edge set $\{e: \ Z_e < p\}$. Since the $Z_e$ are i.i.d., this is the random graph $G_{n,\hat p}$, with $\hat p = \Pr\brac{Z_e < p}$. Since $Z_e \leq 1 + \lambda$, $\hat p = 1$ for $p > 1+\lambda$, so the last integral can be taken from $0$ to $1+\lambda$ and after a change of variables $p\gets \frac{p}{1+\l}$, we get
\begin{equation}\label{eq:EL}
\E L_n = (1+\lambda)\int_0^1 \E \big(\kappa(G_{n,\hat p(p)}) - 1\big) \dd p,
\end{equation}
where
\begin{align*}
\hat p(p) = \Pr(Z_e < (1+\lambda)p) &= \Pr\brac{\frac{1}{1+\lambda} W_e + \frac{1}{1+\lambda^{-1}} C_e < p} \\
&= \left|\left\{(u,v) \in [0,1]^2, \ \frac{1}{1+\lambda}u+\frac{1}{1+\lambda^{-1}} v \leq p\right\}\right|
\end{align*}
where in the last expression $|\cdot|$ denotes Lebesgue measure. An elementary computation (given in an appendix) yields
\begin{equation}\label{eq:phat}
\hat p(p) = \begin{cases} \frac{(1+\lambda)(1+\lambda^{-1})}{2}p^2, & \qquad\qquad\ \ \ 0 \leq p \leq \frac{1}{1+\max\{\lambda,\lambda^{-1}\}} \\
-\frac12\min\{\lambda,\lambda^{-1}\} + p(1+\min\{\lambda,\lambda^{-1}\}), & \frac{1}{1+\max\{\lambda,\lambda^{-1}\}} < p \leq \frac{1}{1+\min\{\lambda,\lambda^{-1}\}} \\
1 - \frac{(1+\lambda)(1+\lambda^{-1})}{2}(1-p)^2, & \frac{1}{1+\min\{\lambda,\lambda^{-1}\}} < p \leq 1\end{cases}
\end{equation}
{\red For convenience, we also include an expression for the inverse function (we need this later when we change variables in integration).
\begin{equation}\label{eq:phat-inv}
p(\hat p) = \begin{cases} \sqrt{\frac{2}{(1+\lambda)(1+\lambda^{-1})}}\sqrt{\hat p}, & \ \qquad\qquad\qquad\qquad\qquad\qquad\ \ \ 0 \leq \hat p \leq \frac{1+\min\{\lambda,\lambda^{-1}\}}{2(1+\max\{\lambda,\lambda^{-1})\}}, \\
\frac{\hat p + \frac{1}{2}\min\{\lambda,\lambda^{-1}\}}{1+\min\{\lambda,\lambda^{-1}\}}, & -\frac{1}{2}\min\{\lambda,\lambda^{-1}\} + \frac{1+\min\{\lambda,\lambda^{-1}\}}{1+\max\{\lambda,\lambda^{-1}\}} < \hat p \leq -\frac{1}{2}\min\{\lambda,\lambda^{-1}\} +1, \\
1-\sqrt{\frac{2}{(1+\lambda)(1+\lambda^{-1})}}\sqrt{1-\hat p}, & \qquad\qquad\quad -\frac{1}{2}\min\{\lambda,\lambda^{-1}\} +1 < \hat p \leq 1.\end{cases}
\end{equation}
}
Now we can proceed with evaluating $\E L_n$ given by \eqref{eq:EL}. First observe that  if {\red $q\in \left[\frac{1000\log n}{n},1\right]$} then we have 
\begin{equation}\label{eq:kappa-triv}
\E \kappa(G_{n,q}) = 1 + o(n^{-200}).
\end{equation}
This is because
\begin{align}
1 \leq \E \kappa(G_{n,q}) &\leq 1 + n\Pr(G_{n,q} \text{ is not connected})\nonumber \\
&\leq 1 + n\sum_{k=1}^{n/2} \binom{n}{k}k^{k-2}q^{k-1}(1-q)^{k(n-k)} \nonumber \\
&\leq 1 + \frac{n}{q}\sum_{k=1}^{n/2} \left(\frac{en}{k}\right)^kk^ke^{-qk(n-k)}\nonumber  \\
&\leq 1 + \frac{n^2}{1000\log n}\sum_{k=1}^{n/2}\left(en e^{-\frac{1000\log n}{n}\frac{n}{2}}\right)^k\nonumber  \\
&\leq 1 + \frac{n^3}{1000\log n}\frac{e}{n^{499}}\nonumber  \\
&= 1 + o(n^{-200}).\label{100}
\end{align}
Therefore we can distinguish the following cases depending on the value of $\lambda$.

\bigskip
\noindent\textbf{Case 1.} $\frac{2000\log n}{n} \leq \lambda \leq \frac{n}{2000\log n}$. Note that then 
$$\hat p \left(\frac{1}{1+\max\{\lambda,\lambda^{-1}\}}\right) = \frac{1}{2}\frac{1+\min\{\lambda,\lambda^{-1}\}}{1+\max\{\lambda,\lambda^{-1}\}} = \frac{1}{2}\min\{\lambda,\lambda^{-1}\} \geq \frac{1000\log n}{n},$$
so by \eqref{eq:kappa-triv}, the integration over the second and third range from \eqref{eq:phat} gives the contribution $(1+\lambda)o(n^{-100})$ in \eqref{eq:EL}. Consequently,
\[
\E L_n = (1+\lambda)\int_0^{\frac{1}{1+\max\{\lambda,\lambda^{-1}\}}} \E \big(\kappa(G_{n,\frac{(1+\lambda)(1+\lambda^{-1})}{2}p^2}) - 1\big) \dd p + (1+\lambda)o(n^{-200}).
\]
By the same reason, we also have
\[
(1+\lambda)\int_{\frac{1}{1+\max\{\lambda,\lambda^{-1}\}}}^{\sqrt{\frac{2}{(1+\lambda)(1+\lambda^{-1})}}} \E \big(\kappa(G_{n,\frac{(1+\lambda)(1+\lambda^{-1})}{2}p^2}) - 1\big) \dd p = \sqrt{2\lambda}o(n^{-200}).
\]
Thus
\begin{align*}
\E L_n &= (1+\lambda)\int_0^{\sqrt{\frac{2}{(1+\lambda)(1+\lambda^{-1})}}} \E \big(\kappa(G_{n,\frac{(1+\lambda)(1+\lambda^{-1})}{2}p^2}) - 1\big) \dd p + (1+\sqrt{2\lambda}+\lambda)o(n^{-200}) \\
&=  (1+\lambda)\int_0^{\sqrt{\frac{2}{(1+\lambda)(1+\lambda^{-1})}}} \E \big(\kappa(G_{n,\frac{(1+\lambda)(1+\lambda^{-1})}{2}p^2}) - 1\big) \dd p + o(n^{-100}).
\end{align*}
Changing the variables yields
\begin{equation}\label{eq:EL2}
\E L_n = \sqrt{\frac{\lambda}{2}}\int_0^{1} \E \big(\kappa(G_{n,q}) - 1\big) \frac{\dd q}{\sqrt{q}} + o(n^{-100}).
\end{equation}
It remains to deal with the integral $\int_0^{1} \E \big(\kappa(G_{n,q}) - 1\big) \frac{\dd q}{\sqrt{q}}$. As before, thanks to \eqref{eq:kappa-triv}, we have
\begin{equation}\label{eq:kappaGnq}
\int_0^{1} \E \big(\kappa(G_{n,q}) - 1\big) \frac{\dd q}{\sqrt{q}} = \int_0^{\frac{1000\log n}{n}} \E \big(\kappa(G_{n,q}) - 1\big) \frac{\dd q}{\sqrt{q}} + o(n^{-100}).
\end{equation}
Decompose
\begin{equation}\label{eq:decomp}
\kappa(G_{n,q}) = \sum_{k=1}^{k_0} A_k + \sum_{k=3}^{k_0} B_k + R,
\end{equation}
where $A_k$ is the number of components which are $k$ vertex trees, $B_k$ is the number of non-tree components on $k$ vertices and $R$ is the number of components on at least $k_0$ vertices. Here we set $k_0 = \log n$.

For the tree components, we have
\begin{equation}\label{eq:Ak-gen}
\E A_k = \binom{n}{k}k^{k-2}q^{k-1}(1-q)^{k(n-k)+\binom{k}{2}-k+1}.
\end{equation}
For $q \leq \frac{1000\log n}{n}$ and $k \leq \log n$, we have $(1-q)^{-k^2+\binom{k}{2}-k+1} \leq e^{qk^2} \leq e^{\frac{1000(\log n)^3}{n}} = 1+o(1)$ and $\binom{n}{k} = (1+o(1))\frac{n^k}{k!}$, hence
\[
\E A_k = (1+o(1))\frac{n^k}{k!}k^{k-2}q^{k-1}(1-q)^{kn}.
\]
Thus
\begin{align}\notag
\int_0^{\frac{1000\log n}{n}} \E \big(\sum_{k=1}^{\log n}A_k - 1\big) \frac{\dd q}{\sqrt{q}} = (1+o(1))\sum_{k=1}^{\log n}\int_0^{\frac{1000\log n}{n}}\frac{n^k}{k!}k^{k-2}q^{k-1}&(1-q)^{kn}\frac{\dd q}{\sqrt{q}} \\
&+ O\left(\sqrt{\frac{\log n}{n}}\right)\label{eq:trees}.
\end{align}
Setting $q = \frac{x}{n}$ gives
\[
\int_0^{\frac{1000\log n}{n}}\frac{n^k}{k!}k^{k-2}q^{k-1}(1-q)^{kn}\frac{\dd q}{\sqrt{q}} = \sqrt{n}\frac{k^{k-2}}{k!}\int_0^{2000\log n} x^{k-1}\left(1 - \frac{x}{n}\right)^{kn} \frac{\dd x}{ \sqrt{x}}.
\]
Using $1-t = e^{-t + O(t^2)}$ as $t \to 0$, for $x \leq 1000\log n$ and $k \leq \log n$, we have $\left(1 - \frac{x}{n}\right)^{kn} = e^{-kx + O(\frac{(\log n)^3}{n})} = (1 + o(1))e^{-kx}$. Therefore
\begin{align*}
\int_0^{\frac{1000\log n}{n}} \E \big(\sum_{k=1}^{\log n}A_k - 1\big) \frac{\dd q}{\sqrt{q}} = (1+o(1))\sqrt{n}\sum_{k=1}^{\log n}\frac{k^{k-2}}{k!}\int_0^{1000\log n}&x^{k-1}e^{-kx}\frac{\dd x}{\sqrt{x}}\\
&+ O\left(\sqrt{\frac{\log n}{n}}\right).
\end{align*}
If the integral was from $0$ to $\infty$, we could express it using the gamma function. Since, crudely $\frac{1}{\sqrt{x}} \leq 1$ on the domain of integration,
\begin{align*}
\sqrt{n}\sum_{k=1}^{\log n}\frac{k^{k-2}}{k!}\int_{1000\log n}^\infty x^{k-1}e^{-kx}\frac{\dd x}{\sqrt{x}} &\leq \sqrt{n} \sum_{k=1}^{\log n}\frac{k^{k-2}}{k!}\int_{1000\log n}^\infty x^{k-1}e^{-kx}\dd x
\end{align*}
and for $k=1$ on the right hand side we get $\sqrt{n}e^{-1000\log n} = o(n^{-900})$, whereas for $k \geq 2$ we get
\begin{align*}
\sqrt{n} &\sum_{k=2}^{\log n}\frac{k^{k-2}}{k!}\int_{1000\log n}^\infty x^{k-1}e^{-x}e^{-(k-1)\cdot 1000\log n}\dd x \\
&\leq O(n^{1001})\sum_{k=2}^{\log n} \frac{k^{k-2}}{k!}(k-1)!n^{-1000k} \\
&\leq O(n^{1001})\sum_{k=2}^{\log n} \left(\frac{k}{n^{1000}}\right)^k = O(n^{-500}).
\end{align*}
We can conclude that
\[
\int_0^{\frac{1000\log n}{n}} \E \big(\sum_{k=1}^{\log n}A_k - 1\big) \frac{\dd q}{\sqrt{q}} = (1+o(1))\sqrt{n}\sum_{k=1}^{\log n}\frac{k^{k-2}}{k!}\int_0^{\infty}x^{k-3/2}e^{-kx}\dd x 
+ O\left(\sqrt{\frac{\log n}{n}}\right).
\]
It remains to compute the sum over $k$. We have
\begin{equation}\label{sum}
\sum_{k=1}^{\log n}\frac{k^{k-2}}{k!}\int_0^{\infty}x^{k-3/2}e^{-kx}\dd x = \sum_{k=1}^{\log n}\frac{k^{k-2}}{k!}\frac{\sqrt{k}}{k^k}\Gamma\left(k-\frac{1}{2}\right) = \sum_{k=1}^{\log n} \frac{1}{k^{3/2}}\frac{\Gamma\left(k-\frac{1}{2}\right)}{k!}.
\end{equation}
Since for $k \geq 3$, $\Gamma(k-1/2) \leq \Gamma(k) = (k-1)!$, the series converges and we have
\begin{equation}\label{eq:Ak}
\int_0^{\frac{100\log n}{n}} \E \big(\sum_{k=1}^{\log n}A_k - 1\big) \frac{\dd q}{\sqrt{q}} = (1+o(1))a_0\sqrt{n},
\end{equation}
where
\begin{equation}\label{eq:c0}
a_0 = \sum_{k=1}^{\infty} \frac{1}{k^{3/2}}\frac{\Gamma\left(k-\frac{1}{2}\right)}{k!}.
\end{equation}

To bound the contribution form non-tree components, note that
\begin{equation}\label{eq:Bk-gen}
\E B_k \leq \binom{n}{k}k^{k}q^{k}(1-q)^{k(n-k)} \leq \big[enqe^{-qn}\big]^ke^{qk^2}.
\end{equation}
Thus
\begin{align*}
\int_0^{\frac{1000\log n}{n}} \E \big(\sum_{k=3}^{\log n}B_k\big) \frac{\dd q}{\sqrt{q}} &\leq e^{\frac{1000\log n}{n}(\log n)^2} \sum_{k=3}^{\log n} \int_0^{\frac{1000\log n}{n}} \big[enqe^{-qn}\big]^k \frac{\dd q}{\sqrt{q}} \\
&\leq (1 + o(1)) (\log n)\int_0^{\frac{1000\log n}{n}} \big[enqe^{-qn}\big]^3 \frac{\dd q}{\sqrt{q}} \\
&= O(\log n) \frac{1}{\sqrt{n}}\int_0^{1000 \log n} x^{5/2}e^{-3x} \dd x, 
\end{align*}
so
\begin{equation}\label{eq:nontrees}
\int_0^{\frac{1000\log n}{n}} \E \big(\sum_{k=3}^{\log n}B_k \big) \frac{\dd q}{\sqrt{q}} = O\left(\frac{\log n}{\sqrt{n}}\right).
\end{equation}

Finally, for the large components, since 
\begin{equation}\label{eq:R-gen}
R \leq \frac{n}{k_0},
\end{equation}
we get $R \leq \frac{n}{\log n}$, so we have
\begin{equation}\label{eq:R}
\int_0^{\frac{1000\log n}{n}} \E \big(R \big) \frac{\dd q}{\sqrt{q}} \leq 2\sqrt{\frac{1000\log n}{n}}\frac{n}{\log n} = O\left(\frac{\sqrt{n}}{\sqrt{\log n}}\right).
\end{equation}

Combining \eqref{eq:Ak}, \eqref{eq:nontrees}, \eqref{eq:R} with \eqref{eq:decomp} and plugging into \eqref{eq:kappaGnq}, we obtain
\[
\int_0^{1} \E \big(\kappa(G_{n,q}) - 1\big) \frac{\dd q}{\sqrt{q}} = (1+o(1))c_0\sqrt{n}.
\]
In view of \eqref{eq:EL2} this gives \eqref{eq:EL-formula-lambda-mid}.

\bigskip
\noindent\textbf{Case 2.} $\lambda < \frac{2000\log n}{n}$. Then plainly $\min\{\lambda,\lambda^{-1}\} = \lambda$ and $\max\{\lambda,\lambda^{-1}\} = \lambda^{-1}$. Since $\hat{p}(p) \geq \hat{p}(\frac{1}{1+\lambda}) = 1 - \frac{\lambda}{2}$, for $p \geq \frac{1}{1+\lambda}$, in view of \eqref{eq:kappa-triv}, the third range in \eqref{eq:phat}, that is $\frac{1}{1+\lambda} < p \leq 1$, gives the contribution $(1+\lambda)o(n^{-200}) = o(n^{-200})$ in \eqref{eq:EL}. For the remaining two ranges, changing the variables $q = \hat{p}(p)$ in \eqref{eq:EL} gives
\[
\E L_n = \sqrt{\frac{\lambda}{2}}\int_0^{\lambda/2}\E\Big[\kappa(G_{n,q})-1\Big] \frac{\dd q}{\sqrt{q}} + \int_{\lambda/2}^{1 - \lambda/2}\E\Big[\kappa(G_{n,q})-1\Big] \dd q + o(n^{-100}).
\]
By \eqref{eq:kappa-triv}, for the second integral we get
\[
\int_{\lambda/2}^{1 - \lambda/2}\E\Big[\kappa(G_{n,q})-1\Big] \dd q = \int_{\lambda/2}^{\frac{1000\log n}{n}}\E\Big[\kappa(G_{n,q})-1\Big] \dd q + o(n^{-200}),
\]
so
\begin{equation}\label{eq:EL-case2}
\E L_n = \sqrt{\frac{\lambda}{2}}\int_0^{\lambda/2}\E\Big[\kappa(G_{n,q})-1\Big] \frac{\dd q}{\sqrt{q}} + \int_{\lambda/2}^{\frac{1000 \log n}{n}}\E\Big[\kappa(G_{n,q})-1\Big] \dd q + o(n^{-100}).
\end{equation}
We again decompose $\kappa(G_{n,q})$ as in \eqref{eq:decomp}. Here we set $k_0 = (\log n)^2$. First we show that the $B_k$ and $R$ have small contribution in the integrals above. By \eqref{eq:Bk-gen}, 
\begin{align*}
\sqrt{\frac{\lambda}{2}}\int_0^{\lambda/2}\E\Big[\sum_{k=3}^{k_0} B_k\Big] \frac{\dd q}{\sqrt{q}} &\leq \sum_{k=3}^{k_0} \sqrt{\frac{\lambda}{2}}\int_0^{\lambda/2} \big[enqe^{-qn}\big]^ke^{qk^2} \frac{\dd q}{\sqrt{q}} \\
&\leq e^{\frac{\lambda}{2}k_0^2}\sum_{k=3}^{k_0} \sqrt{\frac{\lambda}{2}}\int_0^{\infty} \big[enqe^{-qn}\big]^k \frac{\dd q}{\sqrt{q}} \\
&\leq e^{\frac{1000(\log n)k_0^2}{n}}\sqrt{\frac{\lambda}{2}}\sum_{k=3}^{k_0} \frac{1}{\sqrt{n}}\int_0^{\infty} \big[exe^{-x}\big]^k \frac{\dd x}{\sqrt{x}} \\
&\leq e^{\frac{1000(\log n)k_0^2}{n}}\sqrt{\frac{1000\log n}{n}} \frac{k_0}{\sqrt{n}}\int_0^{\infty} \big[exe^{-x}\big]^3 \frac{\dd x}{\sqrt{x}} \\
&= O\left(\frac{(\log n)^{5/2}}{n}\right).
\end{align*}
and similarly
\begin{align*}
\int_{\lambda/2}^{\frac{1000\log n}{n}}\E\Big[\sum_{k=3}^{k_0} B_k\Big] \dd q &\leq \sum_{k=3}^{k_0}\int_{\lambda/2}^{\frac{1000\log n}{n}} \big[enqe^{-qn}\big]^ke^{qk^2}\dd q \\
&\leq e^{\frac{1000(\log n)k_0^2}{n}}\frac{k_0}{n}\int_0^\infty \big[exe^{-x}\big]^3 \dd x \\
&=O\left(\frac{(\log n)^{2}}{n}\right).
\end{align*}
By \eqref{eq:R-gen},
\begin{align*}
\sqrt{\frac{\lambda}{2}}\int_0^{\lambda/2}\E R \frac{\dd q}{\sqrt{q}} + \int_{\lambda/2}^{\frac{1000 \log n}{n}}\E R \ \dd q &\leq \frac{n}{k_0} \left(\sqrt{\frac{\lambda}{2}}\int_0^{\lambda/2} \frac{\dd q}{\sqrt{q}} + \int_{\lambda/2}^{\frac{1000 \log n}{n}} \dd q\right) \\
&\leq \frac{n}{k_0}\left( \lambda + \frac{1000\log n}{n} \right) \\
&= O\left(\frac{1}{\log n}\right).
\end{align*}
Putting the last three estimates together with \eqref{eq:EL-case2} yields
\begin{align}
\E L_n &= \sqrt{\frac{\lambda}{2}}\int_0^{\lambda/2}\E\Big[\sum_{k=1}^{k_0}A_k-1\Big] \frac{\dd q}{\sqrt{q}} + \int_{\lambda/2}^{\frac{1000 \log n}{n}}\E\Big[\sum_{k=1}^{k_0}A_k-1\Big] \dd q + O\left(\frac{1}{\log n}\right) \notag\\\label{eq:EL-case2-clean}
&=\sqrt{\frac{\lambda}{2}}\int_0^{\lambda/2}\E\Big[\sum_{k=1}^{k_0}A_k\Big] \frac{\dd q}{\sqrt{q}} + \int_{\lambda/2}^{\frac{1000 \log n}{n}}\E\Big[\sum_{k=1}^{k_0}A_k\Big] \dd q + O\left(\frac{1}{\log n}\right).
\end{align}
Using \eqref{eq:Ak-gen} and repeating verbatim the arguments following it to bound $1-q$, to change the variables $q = \frac{x}{n}$ and to replace $\left(1 - \frac{x}{n}\right)^{kn}$ with $e^{-kx}$, we obtain
\begin{align*}
\E L_n = (1+o(1))\sum_{k=1}^{k_0}\frac{k^{k-2}}{k!}\Big[\sqrt{\frac{\lambda n}{2}}\int_{0}^{\frac{\lambda n}{2}}x^{k-3/2}e^{-kx} \dd x+ \int_{\frac{\lambda n}{2}}^{1000\log n}&x^{k-1}e^{-kx} \dd x\Big] \\
&+ O\left(\frac{1}{\log n}\right).
\end{align*}
As in Case 1, $\sum_{k=1}^{k_0}\frac{k^{k-2}}{k!}\int_{1000\log n}^{\infty}x^{k-1}e^{-kx} \dd x = O(n^{-100})$, so we can replace the integral $\int_{\frac{\lambda n}{2}}^{1000\log n}x^{k-1}e^{-kx} \dd x$ with$ \int_{\frac{\lambda n}{2}}^{\infty}x^{k-1}e^{-kx} \dd x$. Moreover, crude estimates show that
\begin{align*}
&\sum_{k=k_0}^{\infty}\frac{k^{k-2}}{k!}\Big[\sqrt{\frac{\lambda n}{2}}\int_{0}^{\frac{\lambda n}{2}}x^{k-3/2}e^{-kx} \dd x+ \int_{\frac{\lambda n}{2}}^{\infty}x^{k-1}e^{-kx} \dd x\Big] \\
&\leq \sum_{k=k_0}^{\infty}\frac{k^{k-2}}{k!}\Big[\sqrt{1000\log n}\int_{0}^{\infty}x^{k-3/2}e^{-kx} \dd x + \int_{0}^{\infty}x^{k-1}e^{-kx} \dd x\Big] \\
&= \sum_{k=k_0}^{\infty}\frac{k^{k-2}}{k!}\Big[\sqrt{1000\log n}\frac{\Gamma(k-\frac{1}{2})}{k^{k-1/2}} + \frac{\Gamma(k)}{k^k}\Big] \\
&\leq \sqrt{1000\log n}\sum_{k=k_0}^{\infty} k^{-5/2} + \sum_{k=k_0}^{\infty} k^{-3} = O\left(\frac{1}{(\log n)^{5/2}}\right).
\end{align*}
Thus finally
\begin{align*}
\E L_n = (1+o(1))\sum_{k=1}^{\infty}\frac{k^{k-2}}{k!}\Big[\sqrt{\frac{\lambda n}{2}}\int_{0}^{\frac{\lambda n}{2}}x^{k-3/2}e^{-kx} \dd x + \int_{\frac{\lambda n}{2}}^{\infty}&x^{k-1}e^{-kx} \dd x\Big] \\
&+ O\left(\frac{1}{\log n}\right).
\end{align*}
Note that in the first integral, we have $\sqrt{\frac{\lambda n}{2}}\frac{1}{\sqrt{x}} \geq 1$, hence the main term (the sum over $k$) is lower-bounded by $\sum_{k=1}^{\infty}\frac{k^{k-2}}{k!}\int_0^{\infty}x^{k-1}e^{-kx} \dd x = \zeta(3)$ and consequently, the $O\left(\frac{1}{\log n}\right)$ term can be incorporated into the $o(1)$ term, which gives \eqref{eq:EL-formula-lambda-small}.

\bigskip
\noindent\textbf{Case 3.} $\lambda > \frac{n}{2000\log n}$. Then plainly $\min\{\lambda,\lambda^{-1}\} = \lambda^{-1}$ and $\max\{\lambda,\lambda^{-1}\} = \lambda$. Changing the variables $q = \hat{p}(p)$ in \eqref{eq:EL} yields
\begin{align*}
\E L_n &= \sqrt{\frac{\lambda}{2}}\int_0^{\frac{1}{2\lambda}}\E\Big[\kappa(G_{n,q})-1\Big] \frac{\dd q}{\sqrt{q}} \\
&+ \lambda\int_{\frac{1}{2\lambda}}^{1 - \frac{1}{2\lambda}}\E\Big[\kappa(G_{n,q})-1\Big] \dd q \\
&+ \sqrt{\frac{\lambda}{2}}\int_{1-\frac{1}{2\lambda}}^{1}\E\Big[\kappa(G_{n,q})-1\Big] \frac{\dd q}{\sqrt{1-q}}.
\end{align*}
Since $1 - \frac{1}{2\lambda} \geq \frac{1000 \log n}{n}$, in view of \eqref{eq:kappa-triv}, the third integral gives
\[
\sqrt{\frac{\lambda}{2}}\int_{1-\frac{1}{2\lambda}}^{1}\E\Big[\kappa(G_{n,q})-1\Big] \frac{\dd q}{\sqrt{1-q}} = o(n^{-200})\sqrt{\frac{\lambda}{2}}\int_{1-\frac{1}{2\lambda}}^{1} \frac{\dd q}{\sqrt{1-q}} = o(n^{-200}).
\]
Similarly, for the second integral we have
\[
\lambda\int_{\frac{1000\log n}{n}}^{1 - \frac{1}{2\lambda}}\E\Big[\kappa(G_{n,q})-1\Big] \dd q = \lambda o(n^{-100})\int_{\frac{1000\log n}{n}}^{1 - \frac{1}{2\lambda}} \dd q = \lambda {\red o(n^{-200})}.
\]
Thus we can write (we incorporate the term $o(n^{-200})$ in $\lambda o(n^{-200})$)
\[
\E L_n = \lambda\left( \sqrt{\frac{1}{2\lambda}}\int_0^{\frac{1}{2\lambda}}\E\Big[\kappa(G_{n,q})-1\Big] \frac{\dd q}{\sqrt{q}} + \int_{\frac{1}{2\lambda}}^{\frac{1000 \log n}{n}}\E\Big[\kappa(G_{n,q})-1\Big] \dd q + o(n^{-200}) \right).
\]
The expression in the bracket is exactly \eqref{eq:EL-case2} with $\lambda$ being replaced by $\lambda^{-1}$. Therefore, from \eqref{eq:EL-formula-lambda-small}, we obtain \eqref{eq:EL-formula-lambda-large}.
\end{proof}
\begin{lemma}\label{rem1}
With the notation of Lemma \ref{lm:Elength}, if $\lambda = O(n)$, we have
\beq{Lnmax}{
\E L_n=\begin{cases}\Theta((\l n)^{1/2})&\text{Case {\red a}}.\\ \Theta(\max\set{1,(\l n)^{1/2}})&\text{Case {\red b}}.\\ \Theta((\lambda n)^{1/2})&\text{Case {\red c}}.\end{cases}
}
and with probability $1-o(n^{-200})$,
\beq{Zmax}{
Z_{\max}=\begin{cases}O\brac{\bfrac{\l\log n}{n}^{1/2}}&\text{Case {\red a}}.\\ O\bfrac{\log n}{n}&\text{Case {\red b}}.\\O(\log n)&\text{Case {\red c}}.\end{cases}
}
where $Z_{\max}=\max\set{Z_e:e\in T^*}$ and $T^*$ is the minimum spanning tree with weights $Z_e$.

Also in Case 3 we have
\beq{Cmax}{
C_{\max}=O\bfrac{(\log n)^{1/2}}{n^{1/2}},
}
where $C_{\max}=\max\set{C_e:e\in T^*}$.
\end{lemma}
\begin{proof}
The claims concerning $\E L_n$ follow directly from \eqref{eq:EL-formula-lambda-mid}, \eqref{eq:EL-formula-lambda-small}, \eqref{eq:EL-formula-lambda-large}.

To justify \eqref{Zmax}, fix $p_0$ and let $X = |\{e \in T^*, \ Z_e > (1+\lambda)p_0\}|$ be the number of edges on the minimum spanning tree having weights $Z_e$ above $(1+\lambda)p_0$. By Janson's formula from \cite{Janson}, $X = \kappa(G_{n,\hat{p}(p_0)})-1$ with $\hat{p}$ given by \eqref{eq:phat}. By the first moment, $\Pr(X > 0) \leq \E X = \E[\kappa(G_{n,\hat{p}(p_0)})-1]$. By \eqref{100}, choosing $p_0$ such that $\hat{p}(p_0) = \frac{1000\log n}{n}$ gives $X = 0$, equivalently $Z_e \leq (1+\lambda)p_0$, with probability $1 - o(n^{-200})$. It remains to bound $(1+\lambda)p_0$. In Case 1, we see from \eqref{eq:phat} that $\frac{(1+\lambda)(1+\lambda^{-1})}{2}p_0^2 = \frac{1000\log n}{n}$, so $(1+\lambda)p_0 = \sqrt{2\lambda}\sqrt{\frac{1000\log n}{n}} = O(\sqrt{\frac{\lambda \log n}{n}})$. In Case 2 we see that we have to use the second formula in \eqref{eq:phat} and $p_0(1+\l) = \frac{1000\log n}{n}+\frac12\l = O(\frac{\log n}{n})$. Similarly in Case 3, $p_0(1+\l^{-1}) = \frac{1000\log n}{n}+\frac1{2\l}$, hence $p_0(1+\l) = O(\log n)$.

For \eqref{Cmax}, we note that $\Pr(W_e\leq q, \ C_e\leq q)=q^2$. Putting $q=(1000\log n/n)^{1/2}$ we see that with the required probability, the random graph $G_{n,q^2}$ is connected. This implies that with the same probability there is a spanning tree $T$ with $Z_e\leq (1+\l)q \ \forall e \in T$. It follows that a spanning tree that minimises $Z$ will have $Z_{\max}\leq (1+\l)q$. (Applying the greedy algorithm will finish before needing an edge with $Z_e>(1+\l)q$.)  So $Z_{\max}\leq (1+\l)q$ and consequently $C_{\max}\leq 2q$.
\end{proof}

\subsection{Concentration}\label{seconc}

{\red
The goal of this section is to prove the following lemma.

\begin{lemma}\label{lm:conc}
For a fixed $\l = O(n)$ and $\e=\frac{1}{\log n}$,
\beq{concphi}{
\Pr(|\f(\l)- \E(\f(\l))|\geq \e\E (L_n(\l))) =o(n^{-100}).
}
\end{lemma}}
\begin{proof}
Recall that $\f(\l)=\min\set{W(T)+\l C(T):\text{$T\in\cT$ }} - \l c_0=L_n(\l)-\l c_0$ (as defined in \eqref{dual}).

In our analysis we consider separately the contribution of long and short edges.  Let $L=n^{1/10}\E(L_n)/n$ and let $Y_L$ denote the total cost of the edges used on the minimum spanning tree with $Z_e\leq L$.  Let $N =\binom{n}{2}$ and note that $Y_L$ is a function of $N$ i.i.d. random variables $X_1,\dots,X_N$.

We will show $Y_L$ is concentrated using a variant of the Symmetric Logarithmic Sobolev Inequality from \cite{BLM03}.  Let $Y_{L,i}'$ denote the same quantity as $Y_L$, but with the variable $X_i$ replaced by an
independent copy $X_i'$.  Then a simplified form of the Symmetric Logarithmic Sobolev Inequality \cite[Corollary 3]{BLM03} says that if
$$\E\brac{\sum_{i=1}^N (Y_L-Y_{L,i}')^2\1_{Y_L>Y_{L,i}'}\big| X_1,\dots,X_N} \leq c$$ 
then for all $t>0$, 
$$\Pr[Y_L > \E{Y_L} + t] \leq e^{-t^2/4c},$$
and if
$$\E\brac{\sum_{i=1}^N (Y_{L,i}'-Y_L)^2\1_{Y_{L,i}'>Y_L}\big| X_1,\dots,X_N} \leq c$$
then for all $t>0$,
$$\Pr(Y_L < \E{Y_L} - t) \leq e^{-t^2/4c}.$$

Changing the value of one edge can change the value of $Y_L$ by at most $L$, so $(Y_L-Y_{L,i}')^2 < L^2$.  Let $I$ denote the indices of the edges which contribute to $Y_L$. If $i\notin I$ then $Y_{L,i}'<Y_L$ implies $X_i' \leq L$.  So
\beq{eqconc1}{
\sum_{i=1}^N (Y_L-Y_{L,i}')^2\1_{Y_L>Y_{L,i}'} \leq \sum_{i \in I}L^2 + \sum_{i \notin I} L^2 \1_{X_i' \leq L}.
}
Now $\Pr(X_i'<L)\leq \Pr(W_e \leq L, \ \l C_e \leq L) \leq L/{\red \Lambda}$ where ${\red \Lambda}=\max\set{\l,1}$. Then, since there are less than $n$ terms in the first sum and less than $n^2$ terms in the second sum, we have
\beq{eqconc2}{
\E\brac{\sum_{i=1}^N (Y_L-Y_{L,i}')^2\1_{Y_L>Y_{L,i}'}\big| X_1,\dots,X_N}\leq L^2n + L^3n^2/{\red \Lambda}.
}
If $i\notin I$ then we also have that $Y_{L,i}'>Y_L$ implies $X_i' \leq L$. So we also have
\beq{eqconc3}{
\E\brac{\sum_{i=1}^N (Y_{L,i}'-Y_L)^2\1_{Y_{L,i}'<Y_L}\big| X_1,\dots,X_N}\leq  L^2n + L^3n^2/{\red \Lambda}.
}

Therefore,
\begin{multline}\label{cond-dev-small}
\Pr\left[|Y_L-\E{Y_L}| \geq \e\E(L_n)\right] \leq 2\exp{\set{-\frac{\e^2\E(L_n)^2}{4(L^2n + L^3n^2/{\red \Lambda})}}}= 2\exp{\set{-\frac{\e^2}{4(n^{2/10}/n + n^{3/10}\E(L_n)/(n{\red \Lambda})})}}\\
\leq 2\exp{\set{-\frac{\e^2}{4(n^{-4/5} + A\frac{n^{-7/10}\max\{1,\sqrt{\l n}\}}{\max\{1,\lambda\}})}}}\leq 2\exp{\set{-\frac{\e^2 n^{1/5}}{A'}} = o(n^{-200})},
\end{multline}
where we have used $\E(L_n)\leq A\max\{1,(\l n)^{1/2}\}$, see Lemma \ref{rem1} and $A, A'$ are universal constants.

Let $Y_L'$ denote the total cost of the edges used with edge cost at least $L$.  We have from Lemma \ref{rem1} that for some $B>0$, with probability $1-o(n^{-200})$,
\beq{Zmax1}{
Z_{\max}\leq \begin{cases}B\bfrac{\l\log n}{n}^{1/2} \leq L=\Theta\left(\frac{n^{1/10}}{n}\sqrt{\l n}\right)&\text{Case 1}.\\ \\\frac{B\log n}{n}\leq L=\Omega\left(\frac{n^{1/10}}{n}\right)&\text{Case 2}.\\ \\ B\log n\leq L=\Theta\left(\frac{n^{1/10}}{n}\sqrt{\l n}\right) = \Omega(n^{1/20})&\text{Case 3}.\end{cases}
}
And so $Y_L'=0$ with probability $1-o(n^{-200})$.
\end{proof}

\subsection{Optimising over $\l$}\label{opt}
The first thing to observe is that $\f$ is a concave function of $\l$, see for example Boyd and Vandenberghe \cite{BV}. This is because it is the minimum of a collection of linear functions.  Ignoring the $(1+o(1))$ factor, it will be differentiable. It follows then that we can maximise $\f(\l)$ by setting its (asymptotic) derivative to zero. {\red On the other hand, by concentration $\phi(\l)$ is close to $\E \phi(\l)$. We first maximize $\E \phi(\lambda)$.

\begin{lemma}\label{lm:maxEphi}
In cases (1), (2), (3) of Theorem \ref{th1}, we respectively have
\beq{case1E}{
\max_{\l}\E\f(\l)=(1+o(1))\frac{c_1^2n}{4c_0},
}
\beq{case2E}{
\max_{\l}\E\f(\l)=(1+o(1))(f(\b^*)-2\a\b^*),
}
\beq{case3E}{
\max_{\l}\E\f(\l)=(n+o(n))\frac{f(\b^*)-\a}{2\b^*}.
}
Moreover, the maximizer $\lambda = \lambda^*$ in each case satisfies $\l^* = O(n)$.
\end{lemma}
}
\begin{proof}
For $\l \in \left[\frac{2000\log n}{n},\frac{n}{2000\log n}\right]$, we have
\[
\E\f(\l) = (1+o(1))c_1\sqrt{\l n}-\l c_0.
\]
Differentiating (ignoring the $(1+o(1))$ term) and setting it to zero we see that $\E\f(\l)$ is maximised at 
\beq{lstar}{
\l^*=(1+o(1))\frac{c_1^2n}{4c_0^2}
} 
and that $\E\f(\l^*)=(1+o(1))\frac{c_1^2n}{4c_0}$. Note that $\l^* \in \left[\frac{2000\log n}{n},\frac{n}{2000\log n}\right]$ for $c_0$ as in (1). This gives \eqref{case1E}.

Now let $c_0=\a n$ where $0<\a<1/2$. We proceed as before.
Putting $\b=\l n/2$ and $c_0=\a n$ into the expression in \eqref{eq:EL-formula-lambda-small} we get 
\[
\E\f(\b)=(1+o(1))\left(\sum_{k=1}^\infty\frac{k^{k-2}}{k!}\b^{1/2}\int_0^\b x^{k-3/2}e^{-kx}\dd x+\int_\b^\infty x^{k-1}e^{-kx}\dd x\right)-2\a\b = (1+o(1))f(\b)-2\a\b.
\]
Differentiating w.r.t. $\b$ we get
\beq{eq1}{
\f'(\b)=(1+o(1))f'(\b)-2\a
}
and hence the solution $\b^*$ to $\f'(\b)=0$ asymptotically satisfies $f'(\b)=2\a$. Clearly $\b^*=\Theta(1)$ which implies that $\l^*={\red\Theta(1/n)}$ and so $\l^*=o(\log n/ n)$ as claimed. Then \eqref{case2E}  follows.

Finally, let $c_0=\a$ where $\a>\z(3)$.
In this case we put $\b=n/2\l$ and proceed as before. Putting $c_0=\a$  into the expression in \eqref{eq:EL-formula-lambda-large} we get 
\[
\E\f(\b)=(n+o(n))\brac{\frac{1}2\sum_{k=1}^\infty\frac{k^{k-2}}{k!}\b^{-1/2}\int_0^\b x^{k-3/2}e^{-kx}\dd x+\b^{-1}\int_\b^\infty x^{k-1}e^{-kx}\dd x}-\frac{\a n}{2\b}=(n+o(n))\frac{f(\b)}{2\b}-\frac{\a n}{2\b}.
\]
Differentiating w.r.t. $\b$ we get
\[
\f'(\b)=(n+o(n))\brac{\frac{f'(\b)}{2\b}-\frac{f(\b)}{2\b^2}}+\frac{\a n}{2\b^2}
\]
and hence the solution to $\f'(\b)=0$ asymptotically satisfies $f(\b)-\b f'(\b)=\a$. Clearly $\b^*=\Theta(1)$ which implies that $\l^*={\red \Theta(n)}$. Then \eqref{case3E} follows.
\end{proof}
To finish, we divide the interval $I=[0,Cn]$ (with $C$ being an appropriate universal constant) into $n^5$ sub-intervals of equal length less than $n^{-3}$. Suppose that the $i$th interval is $[\l_i,\l_{i+1}]$. We observe that for any spanning tree $T$ we have that for $\l\in[\l_i,\l_{i+1}]$,
\[
|(W(T)+\l_iC(T))-(W(T)+\l C(T))| {\red = |\lambda_i - \lambda|C(T) \leq \frac{1}{n^2}}
\]
and so 
\beq{close}{
|\f(\l_i)-\f(\l)|\leq {\red \frac{1}{n^2}} + c_0|\l_i-\l| \leq {\red \frac{2}{n^2}}.
}
So, maximising $\f$ over $\l_1,\l_2,\ldots,\l_{{\red n^5}}$ makes an error in maximising $\f(\l)$ over $I$ of at most ${\red 2n^{-2}}$. 

Using the concentration result \eqref{concphi} of Section \ref{seconc}, we see that for a fixed $\l=\l_i$, there is $\e'$ with $|\e'| \leq \e$ such that we have
\beq{fa}{
\f(\l)= \E \f(\l) + \e'\E L_n = (1+\e')\E L_n - \l c_0 = (1+o(1))c_1\sqrt{\l n}-\l c_0\text{ with probability }1-o(n^{-200}).
}
We see therefore that w.h.p. the expression for $\l=\l_i$ in \eqref{fa} holds simultaneously for all $i=1,2,\ldots,n^5$. Therefore, by Lemma \ref{lm:maxEphi}, we obtain in Case (1), (2), (3) of Theorem \ref{th1}, respectively that
\beq{case1}{
\max_{\l}\f(\l)=(1+o(1))\frac{c_1^2n}{4c_0},
}
\beq{case2}{
\max_{\l}\f(\l)=(1+o(1))(f(\b^*)-2\a\b^*),
}
where $\b^*$ is the unique solution to $f'(\b)=2\a$ (see \eqref{beta}, \eqref{Wstar1}) and
\beq{case3}{
\max_{\l}\f(\l)=(n+o(n))\frac{f(\b^*)-\a}{2\b^*}.
}
One final point. Our expressions for $\f(\l)$ are only valid within a certain range. But because, $\f$ is concave and we have a vanishing derivative, we know that the values outside the range cannot be maximal.

\section{Proof of Theorem \ref{th1}}\label{secthm}
We will use {\red Theorem 3.1} from Goemans and Ravi \cite{GR}:
\begin{theorem}[{\red \cite{GR}}]\label{th2}
There exists a spanning tree $\tT$ such that $W(\tT)\leq \f(\l^*)\leq  W^*$ and $C(\tT)\leq {\red c_0}+C_{\max}(\tT)$, where $C_{\max}(\tT)$ is the maximum cost of an edge of $\tT$.
\end{theorem}
For {\red Cases a and b} from Lemma \ref{lm:Elength} we let $\hc=c_0-\delta$ where $\delta=\frac{2}{\lambda^*}B R_{\ref{Zmax}}$ where $B$ is a suitable hidden constant for \eqref{Zmax} and $R_{\ref{Zmax}}$ is the RHS of \eqref{Zmax}. Suppose now that we replace $c_0$ by $\hc$ and let $\hW$ denote the minimum weight of a tree with cost at most $\hc$. Applying Theorem \ref{th2} we obtain a spanning tree $\hT$ such that $W(\hT)\leq \f(\hat{\l})\leq \hW$ and $c(\hT)\leq \hc+\frac{1}{\lambda^*}BR_{\ref{Zmax}}\leq c_0$. It only remains to show that w.h.p. $\f(\hat{\l})\approx W^*$. This follows from our expressions for $\f(\l^*)$ in Section \ref{opt} and the fact that $\hc\approx c_0$, which we verify now.

In {\red Case a} we have from \eqref{lstar} that,
\[
\frac{\delta}{c_0}\leq O\brac{ \sqrt{\frac{\log n}{\lambda^* nc_0^2}}} = O\brac{ \frac{\sqrt{\log n}}{n} } =o(1).
\]
In {\red Case b} we have $\delta=O\brac{\frac{\log n}{\lambda^* n}}$, $c_0=\Omega(n)$, ${\red\lambda^* = \Omega(\frac{1}{n})}$ and so $\delta/c_0=O\brac{\frac{\log n}{n}} = o(1)$. 

For {\red Case c} we let $\delta=1/\log n$ and proceed as above. We find that  once again $\f(\hat{\l})\approx W^*$ because of the expression \eqref{case3} for $\f(\l^*)$ in Section \ref{opt} and the fact that $\hc\approx c_0$.
We then use Theorem \ref{th2} and \eqref{Cmax} to show that 
\[
C(\hT)\leq \hc+O\brac{\bfrac{\log n}{n}^{1/2}}= c_0-\frac{1}{\log n}+O\brac{\bfrac{\log n}{n}^{1/2}}\leq c_0.
\]
This completes the proof of Theorem \ref{th1}.

\section{More general distributions}\label{g<1}
We now consider the case where we have $W_e,C_e,e\in E(K_n)$ distributed as independent copies of $U^\g,\g<1$, $U \sim \text{Unif}([0,1])$. We follow the same ideas as for $\g=1$, but there are technical difficulties. Let us first though explain the need for the lower bound on $c_0$ in Theorem \ref{thm:main-resx}, up to a logarithmic factor. 
\begin{lemma}\label{lowerbound}
Let $X_1,X_2,\ldots,X_n$ be independent copies of $U^\g$ and let $Y=min_{i\leq n}X_i$. Then 
\beq{minX}{
\E Y\approx \G(\g+1)n^{-\g}.
}
\end{lemma}
\begin{proof}
\begin{align*}
\E\min_{i\leq n}X_i &=\int_{t=0}^1 \Pr(X_1>t^{1/\g})^n \dd t\\
&=\int_{t=0}^1(1-t^{1/\g})^n\dd t\\
&=\g\int_{t=0}^1(1-s)^ns^{\g-1}\dd s\\
&=\g B(n+1,\g)\qquad \text{ Beta distribution}\\
&=\frac{\G(n+1)\G(\g+1)}{\G(n+\g+1)}\\
&\approx \G(\g+1)\frac{(n/e)^n}{((n+\g)/e)^{n+\g}}\qquad \text{Stirling's approximation}\\ 
&= \frac{\G(\g+1)e^\g}{(n+\g)^\g} \bfrac{n}{n+\g}^n\\
&\approx \frac{\G(\g+1)}{n^\g}.
\end{align*}
\end{proof}
It follows from \eqref{minX} that the expected weight of a minimum spanning tree is $\Omega(n^{1-\g})$. To see this, orient the edges of the minimum weight spanning tree away from vertex 1. Associate each edge with its tail (closest to vertex 1). Then each edge has expected weight at least that given in Lemma \ref{lowerbound}. 

We can use the argument of Section \ref{seconc} with $L=n^{\g/4-1}\E(L_n)$ to show concentration around the mean. Because $\Pr(U^\g\leq L)\leq L^{1/\g}$, the R.H.S.'s of \eqref{eqconc2}, \eqref{eqconc3} become $L^2n+L^{2+1/\g}n^2$. Consequently \eqref{cond-dev-small} becomes 
\beq{eqconc4}{
\Pr\left[|Y_L-\E{Y_L}| \geq \e\E(L_n)\right] \leq 2\exp{\set{-\frac{\e^2\E(L_n)^2}{4(L^2n + L^{2+1/\g}n^2)}}}= 2\exp\set{-\frac{\e^2}{4(n^{ \g/2-1}+n^{\g/2+1/4-1/\g}\E(L_n)^{1/\g})}}.
}
Now if $p_0=\bfrac{1000\log n}{n}^\g$ then $\Pr(U^\g\leq p_0)=\frac{1000\log n}{n}$. So, with probability $1-o(n^{-900})$, the edges of weight at most $p_0$ induce a connected graph and we have that $\E(L_n)=O(n^{1-\g}\log n)$. Plugging this into \eqref{eqconc4} we see that 
\multstar{
\Pr\left[|Y_L-\E{Y_L}| \geq \e\E(L_n)\right] \leq \exp\set{-\frac{\e^2}{4(n^{\g/2-1}+n^{\g/2 - 3/4} \log^{ 1/\g} n)}}\\
 \leq \exp\set{-\frac{\e^2}{4(n^{-1/2}+n^{-1/4}\log^{1/\g} n)}}=o(n^{-200}).
}
We have $L=\Omega(n^{\g/4-1}\times n^{1-\g}=n^{-3\g/4})\gg p_0$ and so $Y_L'=0$ with probability $1-o(n^{-900})$. In conclusion, $L_n=\Omega(n^{1-\g})$ w.h.p.

We now turn to estimating the dual value, the equivalent of Lemma \ref{lm:Elength}.
\subsection{Expectation}\label{secexpect-alpha}
In this section, we estimate the expected weight of the minimum spanning tree with edge weights $U_1^\g+U_2^\g$ for independent copies $U_1,U_2$ of $U$.
\begin{lemma}\label{lm:Elength-alpha}
Let $\g \in (0,1)$, $\lambda \geq 0$ and let {\red $L_n=L_n(\l)$} be the total weight of a minimum spanning tree in the complete graph on $n$ vertices with each edge $e$ having weight $Z_e = W_e^\g + \lambda C_e^\g$, where $W_e$ and $C_e$ are i.i.d. copies of $U$. Assuming 
\beq{lrange}{
\left(\frac{1000\log n}{n}\frac{\Gamma(2/\g+1)}{\Gamma(1/\g+1)^2}\right)^\g \leq \lambda \leq \left(\frac{n}{1000\log n}\frac{\Gamma(1/\g+1)^2}{\Gamma(2/\g+1)}\right)^\g,
}
we have
\begin{equation}\label{eq:EL-formula-lambda-mid-alpha}
\E L_n \approx C_\g\l^{\frac{1}{2}}n^{1-\frac{\g}{2}},
\end{equation}
where
\begin{equation}\label{eq:c1-alpha}
C_\g= \frac{\g}2\frac{\Gamma(2/\g+1)^{\g/2}}{\Gamma(1/\g+1)^\g} \sum_{k=1}^\infty \frac{\Gamma(k+\g/2-1)}{k^{\g/2+1}k!}.
\end{equation}
The implied $o(1)$ terms in the above expressions can be taken to be independent of $\l$. Also, we have not optimised all constants.
\end{lemma}
\begin{proof}
We follow closely the proof of Lemma \ref{lm:Elength} which concerns $\g = 1$. 
Janson's formula \eqref{formula} gives
\begin{equation}\label{eq:EL-alpha}
\E L_n = \int_0^{1+\l} \E \big(\kappa(G_{n,\hat p(t)}) - 1\big) \dd t,
\end{equation}
where
\begin{align*}
\hat p(t) = \Pr(W_e^\g + \l C_e^\g < t) = \left|\left\{(u,v) \in [0,1]^2, \ u^\g+\l v^\g < t\right\}\right|.
\end{align*}
{\bf Case 1, $\l \geq 1$:} 
\[
\hat p(t) = \int_{0}^{\min\set{1,t^{1/\g}}}\min\set{1, \left(\frac{t-u^\g}{\l}\right)^{1/\g}} \dd u.
\]
If $t\leq1$ then
\beq{defphat}{
\hat p(t) = \int_{0}^{t^{1/\g}} \left(\frac{t-u^\g}{\l}\right)^{1/\g} \dd u= \frac{t^{2/\g}}{\l^{1/\g}}\frac{\Gamma(1/\g+1)^2}{\Gamma(2/\g+1)}.
}
Let 
\beq{t0}{
\text{$t_0 \in (0,1)$ be such that $\hat p(t_0) = \frac{1000\log n}{n}$, that is $t_0 = \l^{1/2}\left(\frac{1000\log n}{n}\right)^{\g/2} \left(\frac{\Gamma(2/\g+1)}{\Gamma(1/\g+1)^2}\right)^{\g/2}$}
}
 (our assumption on $\l$ is chosen such that this is possible, i.e. this value of $t_0$ is less than one). Then, thanks to \eqref{eq:kappa-triv},
\begin{align}
\E L_n & = \int_0^{t_0}\E \big(\kappa(G_{n,\hat p(t)}) - 1\big) \dd t + \int_{t_0}^{1+\l}\E \big(\kappa(G_{n,\hat p(t)}) - 1\big) \dd t = \int_0^{t_0}\E \big(\kappa(G_{n,\hat p(t)}) - 1\big) \dd t + (1+\l)o(n^{-200})\nonumber\\
\noalign{Change of variables $q = \hat p(t)$, use \eqref{defphat},}\nonumber\\
&=\frac{ \l^{1/2}\g}{2}\frac{\Gamma(2/\g+1)^{\g/2}}{\Gamma(1/\g+1)^\g}\int_0^{\frac{1000\log n}{n}} \E \big(\kappa(G_{n,q}) - 1\big)q^{\g/2-1} \dd q+ (1+\l)o(n^{-200}).\label{valLn}
\end{align}
It remains to handle the last integral. Repeating verbatim all the computations of Lemma \ref{lm:Elength} from \eqref{eq:decomp} to \eqref{eq:R} (the only difference being that $q^{-1/2}$ is replaced by $q^{\g/2 -1}$ in the integrand), we get
\[
\int_0^{\frac{1000\log n}{n}} \E \big(\kappa(G_{n,q}) - 1\big)q^{\g/2-1} \dd q = (1+o(1))a_{0,\g}n^{1-\g/2} + O\left(\left(\frac{\log n}{n}\right)^{\g/2}\right) + O\left(\frac{\log n}{n^{\g/2}}\right) + O\left(\left(\frac{n}{\log n}\right)^{1-\g/2}\right),
\]
where the error terms come from appropriate changes in \eqref{eq:trees} \eqref{eq:nontrees}, \eqref{eq:R}. The constant $a_{0,\g}$ comes from \eqref{sum} and equals
\begin{equation}\label{eq:a0-alpha}
a_{0,\g} = \sum_{k=1}^\infty \frac{\Gamma(k+\g/2-1)}{k^{\g/2+1}k!}.
\end{equation}
Plugging this back into \eqref{valLn}, we conclude that
\[
\E L_n \approx C_\g\l^{1/2}n^{1-\g/2}
\]
with
\[
C_\g = \frac{\g}2\frac{\Gamma(2/\g+1)^{\g/2}}{\Gamma(1/\g+1)^\g}\sum_{k=1}^\infty \frac{\Gamma(k+\g/2-1)}{k^{\g/2+1}k!}.
\]

{\bf Case 2, $\lambda < 1$:} We set $t = \l t'$ in \eqref{eq:EL-alpha} which yields
\[
\E L_n = \l\int_0^{1+1/\l} \E \big(\kappa(G_{n,\hat p(\l t')}) - 1\big) \dd t'
\]
and now $\hat p(\l t') = \Pr(W_e^\g + \l C_e^\g < \l t') = \Pr(C_\e^\g + \frac{1}{\l}W_e^\g < t') = \Pr(W_\e^\g + \frac{1}{\l}C_e^\g < t')$ (because $W_e$ and $C_e$ are assumed to have the same distribution), so using the previously analysed case $\l > 1$ for $\frac{1}{\l}$, we get
\beq{ELn}{
\E L_n \approx \l C_\g\left(\frac{1}{\l}\right)^{1/2}n^{1-\g/2} = C_\g\l^{1/2}n^{1-\g/2}.
}
This completes the proof of the lemma.
\end{proof}
\subsection{Concentration}
We follow the argument of Section \ref{seconc}.
\begin{lemma}
Let $\e=1/\log n$. Then,
\[
\Pr(|\f(\l)- \E(\f(\l))|\geq \e\E (L_n(\l))) =o(n^{-100}).
\]
\end{lemma}
\begin{proof}
Let $L=n^{\g/8-1}\E(L_n)$. We argue that $\Pr(X_i< L)\leq (L/\La)^{1/\g}$ where $\La=\max\set{\l,1}$, giving \eqref{eqconc2} and \eqref{eqconc3} as before. It then follows that
\mult{eqconc5}{
\Pr\left[|Y_L-\E{Y_L}| \geq \e\E(L_n)\right] \leq 2\exp{\set{-\frac{\e^2\E(L_n)^2}{4(L^2n + L^{2+1/\g}n^2/\La^{1/\g})}}}=\\ 2\exp\set{-\frac{\e^2}{4(n^{\g/4-1}+n^{\g/4+1/8-1/\g}\E(L_n)^{1/\g}/\La^{1/\g})}}.
}
Plugging \eqref{ELn} into the RHS of \eqref{eqconc5} and noting that $\l^{1/2}/\La = \min\{\l^{-1/2},\l^{1/2}\} \leq 1$, we obtain
\[
\Pr\left[|Y_L-\E{Y_L}| \geq \e\E(L_n)\right] \leq \exp\set{-\frac{\e^2}{4(n^{\g/4-1}+O(n^{\g/4-3/8}))}}=o(n^{-200}).
\]
Now because $L\approx C_\g \l^{1/2}n^{-3\g/8}\gg t_0$, where $t_0$ is as in \eqref{t0}, we see that $Y_L'=0$ with probability $1-o(n^{-900})$.
\end{proof}
We divide the interval $[0,Cn]$ into $n^5$ sub-intervals as before and optimise the $\f$ by maximising
\[
C_\g\l^{1/2}n^{1-\g/2}-c_0\l.
\]
Solving we get 
\[
\l^*=\bfrac{n^{1-\g/2}C_\g}{2c_0}^2\text{ and }\max_\l \f(\l)=\frac{C_\g^2n^{2-\g}}{4c_0}.
\]
Observe that our assumptions on $c_0$ imply that $\l^*$ satisfies \eqref{lrange}.

After this, we can follow the proof of the case $\g=1$. We only need to check now that the argument of Section \ref{secthm} is still valid. We know that with probability $1-o(n^{-200})$ that $W_e^\g+\l C_e^\g\leq t_0$ for all edges $e$ of the minimum spanning tree. Here $t_0$ is as defined in \eqref{t0} and we note that $t_0/\l^*=o(W^*)$. This follows from
\[
\frac{t_0}{\l^*}=O\bfrac{\log^{\g/2}n}{(\l^*)^{1/2}n^{\g/2}}=O\bfrac{c_0\log^{\g/2}n}{n}\text{ and }W^*=\Omega\bfrac{n^{2-\g}}{c_0}.
\]
We may therefore proceed as in Section \ref{secthm} with $\hc=c_0-t_0/\l^*$ and this completes the proof of Theorem \ref{thm:main-resx}.
\section{Conclusion}
We have determined the asymptotic optimum value to Problem \eqref{prob} w.h.p. The proof is constructive in that we can w.h.p. get an asymptotically optimal solution \eqref{prob} by computing $\hT$ of the previous section. When weights and costs are uniform $[0,1]$, our theorem covers almost all of the possibilities for $c_0$, although there are some small gaps between the 3 cases. Our results for more general distributions have a more limited range and further research is needed to extend this part of the paper. We have also considered more general classes of random variable and here we have a more limited range for $c_0$. 

The present result assumes that cost and weight are independent. It would be more reasonable to assume some positive correlation. This could be the subject of future research. One could also consider more than one constraint, but then we might lose Theorem \ref{th2}.

\appendix
\section{Proof of \eqref{f''}}
We want to show that $h$ is strictly decreasing on $(0,+\infty)$, where 
\beq{fdash}{
h(\b)= \sum_{k=1}^\infty \frac{k^{k-2}}{k!}\beta^{-1/2}\int_0^\beta x^{k-3/2}e^{-kx} \dd x.
}
We have
\[
-2\beta^{3/2}h'(\beta) = \sum_{k=1}^\infty \frac{k^{k-2}}{k!}\Bigg[\int_0^\beta x^{k-3/2}e^{-kx} \dd x - 2\beta^{k-1/2}e^{-k\beta} \Bigg].
\]
Call the right hand side $H(\beta)$. We want to show that it is positive for every $\beta > 0$. We have $H(0) = 0$, so it is enough to show that $H'(\beta)$ is positive for every $\beta > 0$. We have
\[
H'(\beta) = 2\beta^{-1/2}\sum_{k=1}^\infty \frac{k^{k-2}}{k!}\Bigg[k\beta^{k} - (k-1)\beta^{k-1}\Bigg]e^{-k\beta}
\]
and want to show that the sum on the right hand side is positive for every $\beta > 0$. Note that for $\beta \geq 1$, we have $k\beta^{k} - (k-1)\beta^{k-1} > 0$ for every $k\geq 1$, so the sum is positive in this case. Let $0 < \beta < 1$. Separating the first two terms, we rewrite the condition that the sum is positive as
\[
\beta e^{-\beta} + \frac{1}{2}(2\beta^2 - \beta)e^{-2\beta} > \sum_{k=3}^\infty \frac{k^{k-2}}{k!}\Bigg[k-1- k\beta\Bigg]\beta^{k-1}e^{-k\beta}.
\]
Equivalently, multiplying by $\beta^{-1}e^{2\beta}$, we want to show that for every $0 < \beta < 1$,
\[
e^{\beta} + \beta - \frac{1}{2} > \sum_{k=3}^\infty \frac{k^{k-2}}{k!}\Bigg[k-1- k\beta\Bigg]\big(\beta e^{-\beta}\big)^{k-2}.
\]

Let $0 < \beta \leq \frac{2}{5}$. Estimating crudely $k-1-k\beta < k-1$, using $k! > \sqrt{2\pi}k^{k+1/2}e^{-k}$ and then bounding $\frac{k-1}{k^{5/2}} \leq \frac{2}{3^{5/2}}$ for $k \geq 3$, we get
\begin{align*}
\sum_{k=3}^\infty \frac{k^{k-2}}{k!}\Bigg[k-1- k\beta\Bigg]\big(\beta e^{-\beta}\big)^{k-2} &< \frac{2e^2}{3^{5/2}\sqrt{2\pi}} \sum_{k=3}^\infty \big(\beta e^{1-\beta}\big)^{k-2} \\
&= \frac{2e^2}{3^{5/2}\sqrt{2\pi}} \frac{\beta e^{1-\beta}}{1 - \beta e^{1-\beta}}.
\end{align*}
Moreover, we have
\begin{equation}\label{eq:pain1}
\frac{2e^2}{3^{5/2}\sqrt{2\pi}} \frac{\beta e^{1-\beta}}{1 - \beta e^{1-\beta}} < e^{\beta} + \beta - \frac{1}{2}, \qquad 0 < \beta \leq \frac{2}{5},
\end{equation}
(shown below) which finishes the proof in this case.

Let $\frac{2}{5} < \beta < 1$. Estimating crudely $k-1-k\beta < k-1 - \frac{2}{5}k = \frac{3}{5}k-1$, using $k! > \sqrt{2\pi}k^{k+1/2}e^{-k}$ and then bounding $\big(\beta e^{1-\beta}\big)^{k-2} < \beta e^{1-\beta}$ for $k \geq 3$, we get
\begin{align*}
\sum_{k=3}^\infty \frac{k^{k-2}}{k!}\Bigg[k-1- k\beta\Bigg]\big(\beta e^{-\beta}\big)^{k-2} &< \left(\sum_{k=3}^\infty \frac{\frac{3}{5}k-1}{k^{5/2}} \right)\frac{e^2}{\sqrt{2\pi}} \beta e^{1-\beta} \\
&< \frac{3}{5}\frac{e^2}{\sqrt{2\pi}} \beta e^{1-\beta},
\end{align*}
where it can be checked numerically that $\sum_{k=3}^\infty \frac{\frac{3}{5}k-1}{k^{5/2}} < \frac{3}{5}$.
Moreover, we have
\begin{equation}\label{eq:pain2}
\frac{3e^2}{5\sqrt{2\pi}} \beta e^{1-\beta} < e^{\beta} + \beta - \frac{1}{2}, \qquad \frac{2}{5} < \beta < 1,
\end{equation}
(shown below) which finishes the proof in this case.

It remains to prove \eqref{eq:pain1} and \eqref{eq:pain2}.

Showing \eqref{eq:pain1} is equivalent to showing that the function
\[
u(\beta) = \left(e^\beta + \beta - \frac{1}{2}\right)(1-\beta e^{1-\beta}) - \frac{2e^3}{3^{5/2}\sqrt{2\pi}}\beta e^{-\beta}
\]
is positive on $(0,\frac{2}{5})$. We numerically check that $u(\frac{2}{5}) > 0.1$ and it suffices to show that $u$ is decreasing on $(0,\frac{2}{5})$. We find that
\[
e^{\beta}u'(\beta) = e^{2\beta} + (1-e)e^{\beta} + e\beta^2 + \left(\frac{2e^3}{3^{5/2}\sqrt{2\pi}}-\frac{5e}{2}\right)\beta + \frac{e}{2} - \frac{2e^3}{3^{5/2}\sqrt{2\pi}}.
\]
Call the right hand side $\tilde{u}(\beta)$. We have $\tilde{u}(0) < -0.3$ and for $0 < \beta < \frac{2}{5}$,
\begin{align*}
\tilde{u}'(\beta) &= 2e^{2\beta} + (1-e)e^{\beta} + 2e\beta + \frac{2e^3}{3^{5/2}\sqrt{2\pi}}-\frac{5e}{2} \\
&< 2e^{4/5} + 1 - e + \frac{4e}{5} + \frac{2e^3}{3^{5/2}\sqrt{2\pi}}-\frac{5e}{2} < -0.8
\end{align*}
which shows that $\tilde{u}$ decreases, hence $\tilde{u}(\beta)$ is negative, hence $u'(\beta)$ is negative, hence $u$ decreases.

Showing \eqref{eq:pain2} is equivalent to showing that the function
\[
v(\beta) = e^\beta + \beta - \frac{1}{2} - \frac{3e^3}{5\sqrt{2\pi}}\beta e^{-\beta}
\]
is positive on $(\frac{2}{5},1)$. For $\frac{2}{5} < \beta < 1$, we have
\begin{align*}
v'(\beta) &= e^\beta + 1 - \frac{3e^3}{5\sqrt{2\pi}}(1-\beta)e^{-\beta} \\
&> e^{2/5} + 1 - \frac{3e^3}{5\sqrt{2\pi}}\frac{3}{5}e^{-2/5} > 0.5
\end{align*}
(we used that $(1-\beta)e^{-\beta}$ decreases on $(0,2)$). This shows that $v$ increases on $(\frac{2}{5},1)$, hence $v(\beta) > v(\frac{2}{5}) > 0$ for $\frac{2}{5} < \beta < 1$.

\section{Proof of \eqref{eq:phat}}

We need to compute the surface area of the subset $\left\{(u,v) \in [0,1]^2, \ \frac{1}{1+\lambda}u+\frac{1}{1+\lambda^{-1}} v \leq p\right\}$ of the unit square $[0,1]^2$. The line $\frac{1}{1+\lambda}u+\frac{1}{1+\lambda^{-1}} v = p$ intersects the $u$ and $v$ axes respectively at $u_0 = p(1+\lambda)$ and $v_0 = p(1+\lambda^{-1})$. Thus when both $u_0$ and $v_0$ are less than $1$, the subset is a right triangle whose area is $\frac{1}{2}u_0v_0$. This gives the formula in the first case of \eqref{eq:phat}. When exactly one of $u_0$ and $v_0$ is less than $1$ and the other one is greater than $1$, the subset is a trapezoid and computing its area gives the formula in the second case of \eqref{eq:phat}. Finally, if both $u_0$ and $v_0$ are greater than $1$, the subset is the complement of a right triangle and the formula in the third case of \eqref{eq:phat} follows from the first one by changing $p$ to $1-p$ and taking the complement.

\end{document}